%% file: main.tex
\documentclass[english]{ourlematema}
\usepackage{xcolor}
\usepackage{cleveref}
\usepackage{mathtools}

\newtheorem{setup}[thm]{Setup}
\newtheorem{notation}[thm]{Notation}
\newtheorem{procedure}[thm]{Procedure}
\newtheorem{example}[thm]{Example}

\title{How to stab a polytope}
\titlemark{How to stab a polytope}


\author{Sebastian Seemann}
\address{%
KU Leuven\\
\email{sebastian.seemann@kuleuven.be}
}
\author{Francesca Zaffalon}
\address{%
MPI for Mathematics in the Sciences, Leipzig\\
Weizmann Institute of Science\\
\email{francesca.zaffalon@mis.mpg.de}
}

\date{2024/10/15}

\DeclareMathOperator{\Gr}{Gr}

\DeclareMathOperator{\spanv}{span}
\DeclareMathOperator{\conv}{conv}
\DeclareMathOperator{\sign}{sign}

\DeclareMathOperator{\relint}{relint}

\DeclareMathOperator{\cl}{cl}
\DeclareMathOperator{\rowspan}{rowspan}
\DeclareMathOperator{\Id}{Id}
\DeclareMathOperator{\var}{var}

\newcommand{\mF}{\mathcal{F}}
\newcommand{\mS}{\mathcal{S}}

\newcommand{\mA}{\mathcal{A}}

\newcommand{\RR}{\mathbb{R}}
\newcommand{\CC}{\mathbb{C}}
\newcommand{\PP}{\mathbb{P}}
\newcommand{\Pk}{P^{[k]}}
\newcommand{\Pkmax}{P^{[k]}_{\max}}
\newcommand{\HPk}{\mathcal{H}_P^k}
\newcommand{\CPk}{\mathrm{C}_P^k}
\newcommand{\Fl}{\mathscr{F}\!\ell}

\usepackage{color}
\usepackage{graphicx}
\usepackage{xcolor}
\usepackage{soul}
\newcommand{\mathcolorbox}[2]{\colorbox{#1}{$\displaystyle #2$}}

\begin{document}

\maketitle

\begin{abstract}
    \noindent We study the set of linear subspaces of a fixed dimension intersecting a given polytope.
    To describe this set as a semialgebraic subset of a Grassmannian, we introduce a Schubert arrangement of the polytope, defined by the Chow forms of the polytope’s faces of complementary dimension.
    We show that the set of subspaces intersecting a specified family of faces is defined by fixing the sign of the Chow forms of their boundaries. We give inequalities defining the set of stabbing subspaces in terms of sign conditions on the Chow form.
\end{abstract}

\section{Introduction}

The amplituhedron, introduced by Arkani-Hamed and Trnka~\cite{amplituhedron}, has initiated the use of new polyhedral and algebro-geometric methods in the study of scattering amplitudes in quantum field theories. 
For positive integers $k,n,m$ such that $k+m \leq n$, and given a $n\times(k+m)$ totally positive matrix, the amplituhedron $\mA_{n,k,m}$ is defined as the image of the totally non-negative Grassmannian $\Gr_{\geq 0}(k,n)$ under the map $\Tilde{Z}$
\[ \mA_{n,k,m} = \Tilde{Z}(\Gr_{\geq 0}(k,n)) \subseteq \Gr(k,k+m), \]
where $\Tilde{Z}(C) = C\cdot Z$, with $C$ a $k\times n$ matrix representing the point in $\Gr_{\geq 0}(k,n)$.
In particular, each point in the amplituhedron is a subspace intersecting the cyclic polytope defined by $Z$.
The case of $k=2,m=2$ has been proven to be a positive geometry~\cite{Simonampl}, as introduced in~\cite{PositiveGeometryandCanonicalForms}. 
Here the amplituhedron $\mA_{n,2,2}$ is a set of lines in $\PP^3$ that can be written as a product of a totally non-negative line in $\PP^{n-1}$ with a fixed totally positive matrix. This condition translates to lines in $\PP^3$ intersecting the cyclic polytope defined by $Z$ in a specific way.
More generally, loop amplituhedra are sets of linear subspaces intersecting a tree amplituhedron. For one loop and $k=1$, we are once again interested in describing some set of linear subspaces stabbing a cyclic polytope. See Section~\ref{sec : loop amplituhedra} for more details.

In this paper we consider a generalization of these questions. Given a full-dimensional polytope $P$ in projective space $\PP^{n-1}$ and given a positive integer $k<n$, how can we characterize the {\em $k$-stabbing set} $\Pk$ of $k$-dimensional linear subspaces intersecting the polytope $P$?
In order to characterize this set we introduce a hyperplane arrangement in the Grassmannian $\Gr(k,n)$ defined by the Schubert varieties describing subspaces intersecting $(n-k-1)$-dimensional faces of the polytope $P$. We divide the set of linear subspaces intersecting $P$ into stabbing chambers, based on the families of faces of the polytope they intersect. In Theorem~\ref{thm : sign var} we show that each stabbing chamber is defined by some Chow forms having fixed sign. This allows us to provide inequalities defining the semialgebraic set $\Pk$ by requiring that the sign condition is satisfied by the Chow forms of the boundaries of at least one $(n-k-1)$-dimensional face of $P$, Theorem~\ref{thm : inequ}. In Section~\ref{sec : simplicial} we focus on special classes of polytopes and study the relation of our stabbing sets with other important semialgebraic subsets of Grassmannians.

Earlier work studying hyperplane intersections of a given polytope with the goal of optimizing quantities such as the volume or the number of $k$-faces of such a section has been carried out in \cite{BestPolytopeSlice}. The \emph{slicing chambers} studied there are examples of the stabbing chambers we study in Section~\ref{sec : max stabb sign}. Hyperplanes slicing cyclic polytopes have also been studied in the context of the amplituhedron~\cite{m=1Amplituhedron}. Our aim is to generalize both settings to general polytopes being ``stabbed'' with subspaces of any dimension.

\section{Schubert arrangements and Chow forms}

We start by recalling the definitions of the main objects of our interest, Grassmannians and Schubert divisors. We then show how to associate a Schubert arrangement to a polytope. Finally, we study in more detail the loop amplituhedron and explain the relations to the problem we are considering.

\subsection{Grassmannians and Schubert divisors}

In this paper, the field is fixed to be $\RR$. Given a vector space $V\subseteq \RR^n$ of dimension $k$ , the image $\PP(V)\subseteq \PP^{n-1}$ of $V$ under the equivalence map defining the projective space is a linear space of dimension $k-1$. By abuse of notation we denote this space by $V$, but always refer to the dimension in affine space. 

\begin{notation} For any positive integers $k,n$ with $k\leq n$, we use the following notation: $[n]=\{1,\ldots, n\}$, $\binom{[n]}{k} = \{ I\subseteq [n] \text{ such that } |I| =k \}$.
\end{notation}

\vspace{5pt}
The \emph{Grassmannian} $\Gr(k,n)$ is the set of $k$-dimensional linear subspaces of $\RR^n$.
It can be realized as a projective variety inside $\PP^{\binom{n}{k}-1}$ via the Pl\"ucker embedding. Such an embedding can be realized in two ways, corresponding to different parameterizations of points in $\Gr(k,n)$.
\begin{itemize}
    \item If $V\in \Gr(k,n)$ is the row-span of a $k\times n$ matrix $M$, the \emph{dual Pl\"ucker coordinates} are defined as the $k\times k$ maximal minors of the matrix $M$. For $I\in \binom{[n]}{k}$, we will denote by $q_I$ the dual Pl\"ucker coordinate labeled by $I$.
    \item If $V\in \Gr(k,n)$ is the kernel of an $(n-k)\times n$ matrix $N$, the \emph{primal Pl\"ucker coordinates} are defined as the $(n-k)\times (n-k)$ maximal minors of the matrix $N$. For $J\in \binom{[n]}{n-k}$, we will denote by $p_J$ the primal Pl\"ucker coordinate labeled by $J$.
\end{itemize}
In both cases a vectors of Pl\"ucker coordinates are well-defined points in $\PP^{\binom{n}{k}-1}$ and they do not depend on the choice of the matrix. Moreover, primal and dual Pl\"ucker coordinates satisfy the following equality as points in projective space:
\[ (q_I)_{I\in \binom{[n]}{k}} = ((-1)^{\sign(I)} p_{[n]\setminus I})_{I\in \binom{[n]}{k}}, \]
where $\sign(I)= \sign(\sigma_I)$ with $\sigma_I$ such that if $I=\{ i_1<\cdots < i_{k} \}$ and $[n]\setminus I = \{j_1<\cdots < j_{n-k}\}$, then $\sigma_I = (i_1\, \ldots\, i_{k}\, j_1\, \ldots\, j_{n-k})$.

\begin{dfn}
    The \emph{totally non-negative Grassmannian} $\Gr_{\geq 0}(k,n)$ (resp. {\em totally positive Grassmannian} $\Gr_{>0}(k,n)$) is the semialgebraic subset of $\Gr(k,n)$ given by points $V\in \Gr(k,n)$ whose non-zero dual Pl\"ucker coordinates have all the same sign (resp. whose dual Pl\"ucker coordinates are all non-zero and have all the same sign). A {\em totally non-negative matrix} (resp. {\em totally positive matrix}) is any matrix representing an element of the totally non-negative Grassmannian (resp. totally positive Grassmannian).
\end{dfn}

We are interested in studying semialgebraic sets inside the Grassmannian $\Gr(k,n)$. In order to characterize and describe such sets, we will make use of well-studied subvarieties of the Grassmannian, Schubert varieties. An introduction to the topic can be found in \cite{fulton}. 
In particular, we are interested in Schubert divisors in $\Gr(k,n)$, i.e. Schubert varieties of codimension $1$. These are varieties inside $\Gr(k,n)$ given by points intersecting non-trivially a fixed $(n-k)$-dimensional linear space $W$. We denote such a Schubert divisor by $H_W$.

Schubert divisors are defined by one linear equation in the Pl\"ucker coordinates of $\Gr(k,n)$. Indeed, 
\[H_W = \left\{ V\in \Gr(k,n) \mid \det\begin{pmatrix} V\\ W \end{pmatrix}=0 \right\},\]
where we are identifying the spaces $V$ and $W$ with respectively a $k\times n$ and a $(n-k)\times n$ dual representing matrices.
The equation defining the variety $H_W$ can also be seen as the Chow form of the linear space $W$, i.e. the unique equation that vanishes exactly when the input subspace $V$ intersects the fixed subspace $W$. We denote the Chow form of $W$ by $C_W$. 
Suppose $W$ has dual Pl\"ucker coordinates given by $(q_I(W))_{I\in \binom{[n]}{n-k}}$ and let $(p_I)_{I\in \binom{[n]}{n-k}}$ and $(q_{[n]\setminus I})_{I\in \binom{[n]}{n-k}}$ be respectively the primal and the dual Pl\"ucker coordinates on $\Gr(k,n)$. Then, by Laplace expansion, the Chow form $C_W$ has equation
\begin{equation}\label{eq : Chow equation}
    \begin{split}
    C_W &= \sum_{I \in \binom{[n]}{n-k}} (-1)^{\sign(I)} q_I(W) q_{[n]\setminus I} \\
    &= \sum_{I\in \binom{[n]}{n-k}} q_I(W) p_I.
    \end{split}
\end{equation}
where the first polynomial lies in $\CC\bigl[q_J \mid J \in \binom{[n]}{k}\bigr]$ and the second polynomial lies in $\CC\bigl[p_I \mid I \in \binom{[n]}{n-k} \bigr]$. Chow forms are well-studied objects, which 
are defined in general for irreducible projective varieties. For an introduction, see~\cite{chow}.

\begin{dfn}
    A \emph{Schubert arrangement} in $\Gr(k,n)$ is a finite collection of Schubert divisors, $\mathcal{H} = \{H_1,\ldots, H_d\}$.
\end{dfn}

The study of Schubert arrangements was introduced in \cite{Schubertarr} and has applications in algebraic statistics and particle physics.

\subsection{The Schubert arrangement of a polytope}

The main object of study is the set of linear spaces of a given dimension which intersect a projective polytope. In this context, a {\em polytope} $P\subset \PP^{n-1}$ is the convex hull of finitely many points $v_1,\ldots,v_m \in \RR^n$, i.e.
\[ P = \left\{ \sum_{i=1}^m c_i v_i \mid c_i \geq 0 \right\}. \]
Equivalently, $P$ is the image under the projection map $\RR^n \to \PP^{n-1}$ of a cone over a polytope $P'\subset H$, for some hyperplane $H$ not containing the origin.

\begin{dfn}
    Let $P\subseteq \PP^{n-1}$ be a full-dimensional polytope and let $1 \leq k\leq n$. The set of $k$-dimensional linear spaces intersecting the polytope $P$ is denoted by
    \[ P^{[k]} = \{ V\in \Gr(k,n) \mid V\cap P \neq \emptyset \}, \]
    and called the \emph{$k$-stabbing set} of $P$.
\end{dfn}

In the following we will use the terms ``intersecting'' and ``stabbing'' interchangeably, while the term ``slicing'' will be used to denote codimension $1$ spaces intersecting the given polytope. We first study an open subset of $\Pk$, containing spaces intersecting the polytope in a generic way.

\begin{dfn}
    Let $P\subseteq \PP^{n-1}$ be a full-dimensional polytope and let $1 \leq k\leq n$. The set of $k$-dimensional linear subspaces intersecting $P$ only in faces of dimension $n-k$ and bigger is denoted by
    \[ \Pkmax = \{ V\in \Pk \mid V\cap G = \emptyset \text{ for all } G \text{ face of } P \text{ with } \dim(G)<n-k\}. \]
    and is called \emph{maximally $k$-stabbing set} of $P$.
\end{dfn}

We characterize the spaces $\Pk$ and $\Pkmax$ using a natural Schubert arrangement associated to this polytope, defined as follows.

\begin{dfn}
    Let $P\subseteq \PP^{n-1}$ be a full-dimensional polytope and let $1 \leq k\leq n$. The Schubert arrangement 
    \[ \HPk = \{ H_V \mid V=\spanv(F) \text{ for } F \text{ an } (n-k-1)\text{-dimensional face of } P \} \]
    is called the {\em $k$-face Schubert arrangement of $P$}. 
\end{dfn}

Note that $\Pkmax \supseteq \Pk \setminus \HPk$, that is, $k$-dimensional subspaces $V$ that stab the polytope $P$ and are not contained in any Schubert divisor from the $k$-face Schubert arrangement of the polytope are maximally stabbing the polytope.

In particular, $\Pk$ and $\Pkmax$ are characterized in terms of the sign vector of the Chow forms defining the Schubert arrangement. 
Unlike in the case of cyclic polytopes, which arises frequently in applications to physics, general polytopes do not come with a natural order on the vertices. In order to have a consistent sign characterization of the set $\Pk$, we fix the following setup.

\begin{setup}\label{setup : Chow forms}
    Let $P\subseteq \PP^{n-1}$ be a full-dimensional polytope and fix $1\leq k \leq n$. Fix an order $G_1, \ldots, G_f$ on the $(n-k-1)$-dimensional faces of $P$. For each face $G_i$ choose $n-k$ linearly independent vertices $v^{i}_1,\ldots, v^{i}_{n-k}$. Let $M^i$ be the $(n-k)\times n$ matrix containing $v^{i}_1,\ldots, v^{i}_{n-k}$ in its rows, in the given order. Define the Chow form of $G_i$ to be 
    \[ C_{G_i} = \sum_{I \in \binom{[n]}{n-k}} q_I(M^i)p_I. \]
    The \emph{vector of Chow forms} of $P$ is $\CPk = (C_{G_1},
    \ldots, C_{G_f})$.
\end{setup}

With this definition, the vector of Chow forms defines a point in projective space whenever evaluated on a point of $\Gr(k,n)$.

\begin{lemma}\label{lemma : nonzero}
    For $k,n,P$ fixed as above, the vector of Chow forms $\CPk(V) \in \PP^{f-1}$ for every $V\in \Gr(k,n)$, where $f$ is the number of $(n-k-1)$-dimensional faces of $P$.
\end{lemma}
\begin{proof}
We need to show that there is no $V$ such that $\CPk(V) = 0.$ In other words, we need to show that there is no linear space $V$ intersecting every linear span of $(n-k-1)$ faces.
We proceed by induction over $n$. For the base case let $n=3$, so $P \subset \PP^2$ is a planar full dimensional polytope. 
For $k=1$, then $\CPk(V) = 0$ if and only if $V$ is a point in $\PP^n$ contained in all linear span of edges of a $2$ dimensional polytope. 
For $k=2$, $V$ is a line containing all vertices of $P$. Both cases are not possible due to the full-dimensionality of $P$.

Assume by induction that the vector of Chow forms $\CPk(W)$ is not zero for any $\ell<n-1$ dimensional polytope $P$ and for all $W\in \Gr(k,\ell)$ with $k<\ell$.
Suppose by contradiction that there exist some $(n-1)$-dimensional polytope $P$ and $V\in \Gr(k,n)$ with $\CPk(V) = 0$. Since $V$ is not contained in $\spanv(F)$ for every $(n-k-1)$-dimensional face $F$, there exists an $(n-k-1)$-dimensional face $F'$ of $P$ such that $V\cap \spanv (F)$ has dimension smaller than $k$. Then, $\mathrm{C}_{F}^k(V\cap {\spanv}(F))=0$, contradicting the induction hypothesis.

Finally, it is easy to see that the vector $\CPk(V)$ is well-defined in projective space, that is, evaluations of $\CPk$ at different representative of the same linear space $V$ differ by a scalar multiple, by applying \cref{eq : Chow equation}.
\end{proof}

We will be interested in comparing the signs of points in a projective space. Formally, given $v\in \PP^N$ with homogeneous coordinates $(v_0, \ldots, v_N)$, we define $\sign(v) \in \{ +,0,- \}^{N+1}$ to be the vector defined by
\[ \sign(v)_i = \begin{cases}
    + & \text{ if } v_i>0\\
    0 & \text{ if } v_i=0\\
    - & \text{ if } v_i<0.
\end{cases} \]
Let $\alpha,\beta \in \{ +,0,- \}^{N+1}$ be two sign vectors. We denote by $-\alpha$ the sign vector obtained by swapping all the signs. We say that $\alpha \equiv \beta$ if $\alpha = \beta$ or $\alpha = -\beta$. Denote by $\mathrm{S}_N$ the set of sign vectors $\{ +,0,- \}^{N+1}$ modulo the equivalence relation.
The set $\mathrm{S}_N$ has a partial order defined as follows. Two sign vectors $\alpha, \beta \in \mathrm{S}_N$ are such that $\alpha \leq \beta$ if $\alpha_i=\beta_i$ for each $i=0,\ldots,N$ such that $\alpha_i\neq 0$. That is, $\alpha \leq \beta$ if $\alpha$ can be obtained from $\beta$ by adding some zeroes.

\subsection{Loop amplituhedra and stabbed polytopes}\label{sec : loop amplituhedra}

In this section we review in more detail the definitions of tree and loop amplituhedra and explain the relations with the question we are answering. In order to do so, we start by recalling the definition of the loop Grassmannian, as introduced in~\cite{PositiveGeometryandCanonicalForms}.

\begin{dfn}
    Given a sequence $\underline{k}:=(k_1,\dots,k_L)$ of positive integers, the \emph{$L$-loop Grassmannian} $\Gr(k,n;\underline{k})$ is the family of collections of subspaces $V_S \subset \mathbb{R}^n$ indexed by $S$, where $S:=\{s_1,\dots, s_l\}$ is a subset of $\underline{k}$ with $k_S:=k_{s_1}+\dots +k_{s_l}\leq n-k$, such that $\dim V_S=k+k_S$ and $V_S \subset V_{S^{\prime}}$ for $S \subset S^{\prime}$.
    The {\em totally non-negative loop Grassmannian} $\Gr_{\geq 0}(k,n;\underline{k})$ is the subset of collections of totally non-negative subspaces, where a subspace $V \subset \mathbb{R}^n$ is totally non-negative if all its nonzero Pl\"ucker coordinates have the same sign. 
\end{dfn}

This generalization of the totally non-negative Grassmannian allows us to define the loop amplituhedron.

\begin{dfn}
   Let $ \Gr_{\geq 0}(k,n;\underline{k})$ be a totally non-negative loop Grassmannian with $\ell:=k_1=\dots =k_L$. A totally positive $Z \in Mat(n,k+m)$ with $k+L\ell \leq k+m \leq n$ defines a map 
   $$\Tilde{Z}:\Gr_{\geq 0}(k,n;\underline{k}) \to \Gr(k,k+m;\underline{k})$$
   by mapping each subspace individually.
   The {\em loop amplituhedron} $\mathcal{A}_{(k,n,m;\ell,L)}$ is the image of the totally non-negative loop Grassmannian $\Gr_{\geq 0}(k,n;\underline{k})$ under this~map. 

   If we don't require $Z$ to be totally positive, the image of the potentially rational map $\Tilde{Z}$ is called a \emph{loop-Grasstope}.
\end{dfn}

\begin{rem}
    For a totally positive matrix $Z$, the map $\Tilde{Z}$ is well-defined for non-negative subspaces of arbitrary dimension \cite{LamGrasstopes} and respects incidences, so it is a well-defined map into the target loop Grassmannian. More generally, grasstopes have been studied in detail for the $m=1$ case~\cite{grasstopes}.
\end{rem}

The most studied case is the \emph{tree amplituhedron}, i.e. without any loops. As seen in the introduction and studied more in detail in Section~\ref{sec : stabbing cyclic}, amplituhedra can be interpreted as a set of lines stabbing a cyclic polytope. Here we show that higher loop amplituhedra are also related to stabbing sets.

Let us focus on the $1$-loop amplituhedron. The $L=1$ loop Grassmannian $\Gr(k,n;(\ell))$ is also known as the partial flag variety $\Fl(k,k+\ell)$ of partial flags $V_1 \subset V_2$, where $\dim(V_1)=k$ and $\dim(V_2)=k+\ell$.
The loop amplituhedron $\mathcal{A}_{(k,n,m;\ell,1)}$ is then given as a subset of the product of tree amplituhedra $\mathcal{A}_{k,n,m}$ and $\mathcal{A}_{k+\ell,n,m}$. 
For $k=1$ and $L=1$, the $1$-loop amplituhedron $\mathcal{A}_{(1,n,m;\ell,1)}\subset \Fl(1,1+m)\subset \Gr(1,1+m)\times \Gr(1+\ell,1+m)$ is a subset of the incidence correspondence 
\[\mathcal{I}_{P \times P^{[\ell]}} \subset \Gr(1,1+m)\times \Gr(1+\ell,1+m),\]
where $P$ is the cyclic polytope $ \mathcal{A}_{(1,n,m;0,0)}\subset \Gr(1,1+m)\cong \mathbb{P}^m.$
In particular, there is a projection of the loop amplituhedron into the $\ell$-stabbing set of the cyclic polytope 
\[\pi:\mathcal{A}_{(1,n,m;\ell,1)} \to P^{[\ell]}.\]
In order to understand the loop amplituhedron, characterizing the image of $\pi$ is necessary, as well as understanding its fibers.

\begin{example}
Let $k=2$, $n=4$, $m=2$, $L=1$, $\ell=1$ and $Z=\Id$. Then, the image of the $1$-loop amplituhedron $\mathcal{A}_{(1,4,2;1,1)}$ in $\Pk$ consists of totally non-negative lines intersecting the simplex $\Delta \subset \PP^3$.
A element $V$ in $\Gr(k,n)$ is totally non-negative, if and only if every vector $v\in V$ viewed as a sequence of numbers changes sign at most $k-1$ times, \cite[Theorem 1.6]{karp}. Hence the line $V \in \Gr(2,4)$ is totally non-negative, and thus in the image of $\pi$, if and only if it does not pass through points with one of the following sign vectors 
\[(+,+,-,+),(+,-,+,+),(+,-,+,-),(+,-,-,+).\]
\end{example}

\section{Stabbing sets and Chow forms}\label{sec : max stabb sign}

In this section we will describe how sign vectors characterize the sets $\Pk$ and $\Pkmax$. 
First, we restrict to study the set $\Pkmax$. 
We do this by decomposing the set based on how the $k$-subspace intersects the polytope. We then show that each piece of this decomposition can be described by fixing some signs in the vector of Chow forms.

\subsection{Chamber decomposition of $\Pkmax$}

Consider the chamber decomposition of $\Pkmax$ defined as follows. For $V\in \Pkmax$, the associated chamber contains all the $k$-dimensional subspaces of $\RR^n$ that stab the polytope in the same $(n-k)$-dimensional faces. Formally, let 
\begin{align*}
    \mS_P(V) = \left\{ W\in \Gr(k,n) \left| \begin{matrix}W\cap \relint(F) \neq \emptyset \iff V \cap \relint(F) \neq \emptyset \\\text{ for all } (n-k)\text{-dim. faces } F \text{ of } P\end{matrix}\right.\right\}.
\end{align*} 
These sets are called \emph{stabbing chambers} of $P$. Given a stabbing chamber $\mS_P(V)$, we say that $\mS_P(V)$ is \emph{determined} by faces $F_1,\ldots, F_r$ of $P$ if $V$ stabs them and only them among all the $(n-k)$-dimensional faces of $P$. We denote by $\mF_P(V) = \{F_1,\ldots,F_r\}$ the set of determining faces of a stabbing chamber.

Clearly, stabbing chambers define a decomposition of $\Pkmax$, that is, for every $V,W \in \Pkmax$, the corresponding chambers $\mS_P(V)$ and $\mS_P(W)$ are either equal or disjoint.
In Theorem~\ref{thm : sign var} we prove that stabbing chambers of $P$ can be characterized by the signs of some of the Chow forms associated to $P$. Formally, given a stabbing chamber $\mS_P(V)$ with determining faces $\mF_P(V)$, denote by 
\[ \CPk|_{\mF(V)} = (C_G \mid G \text{ is a facet of } F \text{ for some } F\in \mF_P(V)). \]
Fixing the sign of this sub-vector of Chow forms describes the stabbing chamber of $V$.

\begin{thm}\label{thm : sign var}
    The following holds:
    \[ \mS_P(V) = \{ W\in \Gr(k,n) \mid \sign(\CPk|_{\mF(V)}(V)) \equiv \sign(\CPk|_{\mF(V)}(W)) \}. \]
\end{thm}
\begin{proof}
    Let $V\in \Pkmax$ with determining faces of the corresponding chamber given by $\mF_P(V) = \{F_1,\ldots,F_r\}$. 
    Let $W\in \mS_P(V)$, we want to prove that the vector of Chow forms evaluated at $W$ has the same sign vector as the one of $V$ when restricting to the facet of the determining faces.

    The intersection $V\cap P$ is a $(k-1)$-dimensional polytope given as the convex hull of $V\cap F_i = v_i$, for $i=1,\ldots,r\geq k$. Then $V$ can be represented as a $k\times n$ matrix containing $k$ of the vectors $v_i$ as its rows. Without loss of generality, we can assume that these are the first $k$ points $v_i$. We will denote this matrix by $V$. We want to show that we can move the points $v_i$ to the points of intersection $W\cap F_i$ without changing the signs of the Chow forms.

    Let $w_i=W\cap F_i$ and consider the matrix
    \[ U^1_t:=\begin{pmatrix}
    (1-t)v_1 + tw_1 \\
    v_2\\
    \vdots   \\
    v_{k} \\
    \end{pmatrix}. \]
    For every $t\in [0,1]$, the space defined by the matrix $U^1_t$ lies in $\Gr(k,n)$, by convexity of the polytope.
    Hence the map $t \mapsto U^1_t$ is a well-defined path in $\Gr(k,n)$. Moreover, $(1-t)v_1 + tw_1 \in \relint(F_1)$ for each $t\in [0,1]$. The Pl\"ucker coordinates are linear in $t$ along this path. Hence, each Chow form of the facets of $F_1$ is non-zero on each $U_t$, for $t\in[0,1]$. Since the Chow forms are continuous and linear in the Pl\"ucker coordinates, this implies that their sign does not change in the path from $V=U^1_0$ to $U^1_1$. Repeat the same operation for $i=2,\ldots,k$ by defining 
    \[ U^i_t:=\begin{pmatrix}
    w_1 \\
    \vdots\\
    w_{i-1}\\
    (1-t)v_i + tw_i\\
    v_{i+1}\\
    \vdots   \\
    v_{k} \\
    \end{pmatrix}. \]
    In each path we never cross any boundaries of $F_i$, hence the sign of the Chow forms are fixed for every facet of $F_i$. By repeating the same operation for any $k$-subset of the determining faces $F_1,\ldots, F_r$ we obtain the desired result, i.e.
    \[ \sign(\CPk|_{\mF(V)}(V)) = \sign(\CPk|_{\mF(V)}(W)). \]

    \vspace{5pt}
    Suppose now that the sign characterization is true for some $W\in \Gr(k,n)$ and assume by contradiction that $W\not\in \mS_P(V)$. Then there exists some face $F_i\in \mF_P(V)\setminus \mF_P(W)$. Without loss of generality, we can assume that $i=1$. Let $w_i = W\cap \spanv(F_i)$ and define the matrices $U_t^1$ as above. 
    We can choose representative of $w_i$ such that the signs of $C_P(V)|_{\mF(V)}$ and $C_P(W)|_{\mF(V)}$ are the same in affine space.
    By assumption, the intersection of $U_t^1$ with $\spanv(F_1)$ moves outside of $F_1$, hence there exists a $t_0\in [0,1]$ such that $U^1_{t_0}$ intersects a facet $G$ of $F_1$. Since the Pl\"ucker coordinates of $U_t$ are linear in $t$ and the Chow forms are linear in the Pl\"ucker coordinates of $\Gr(k,n)$, $C_G(U^1_t)$ changes sign along the path for $t\in [0,1]$. Since we defined $w_1$ to be the intersection of $W$ with $\spanv(F_1)$, the sign of $C_G$ will stay fixed as we perform the paths defined by $U^2_t, \ldots, U^k_t$. Hence we obtain that $\sign(C_G(W))$ is different (in affine space) from $\sign(C_G(V))$, against the hypothesis. It follows that $W\in \mS_P(V)$.
\end{proof}

Note in particular that there is a dense subset of each stabbing chamber $\mS_P(V)$ which is given by the union of finitely many regions in the complement of the $k$-face Schubert arrangement of $P$.
As a corollary of the proof of the previous theorem we obtain the following result about the topology of the stabbing chambers of $P$.

\begin{cor}
    The closure of the stabbing chambers $\cl(\mS_P(V))$ of the Schubert arrangement $\HPk$ in $\Pk$ are contractible.
\end{cor}
\begin{proof}
 By the proof of \ref{thm : sign var}, for $W \in \mS_P(V)$ and $V\in \Pkmax$ there is a path from $V$ to $W$ where the $\sign(\CPk|_{\mF(V)})$ stays constant. 
\end{proof}

\begin{rem} An interesting problem in the context of Schubert arrangements is counting the number of components in its real region. In general, these regions are not uniquely defined by the sign vector of the defining equations.
See for example \cite[Chapter 5]{Schubertarr}, where it is possible to observe that Theorem~\ref{thm : sign var} might fail if the Schubert arrangement does not come from the faces of a polytope.
For more details about the number of regions in real and complex Schubert arrangement, their connection to \emph{ML-degrees} and a study of their Euler characteristic see \cite{Schubertarr,computingarr}.

\begin{example}[\protect{\cite[Example 5.2]{Schubertarr}}]\label{ex : bad example}
    Consider the four lines $l_1,\dots,l_4 \subset \PP^3$ with respective Chow forms $p_{12},p_{14},p_{23},p_{34}$.
    They are not the linear spans of the edges of a convex polytope, but they are the linear spans of four of the edges of the simplex $\Delta \in \PP^3$, so Theorem~\ref{thm : sign var} does not apply.
    Consider the two lines $V$ and $W$ spanned by 
    
    \[V= \begin{pmatrix}
    1&0& 1&1   \\
    0 &1 & 2 & 1 \\ 
    \end{pmatrix},\qquad  
    W= \begin{pmatrix}
    1&0& 1&-1   \\
    0 &1 & -2 & 1 \\ 
    \end{pmatrix}.\]
     The line $V$ stabs the simplex $\Delta$ through the facets ${134}$ and ${234}$, and $W$ does not intersect $\Delta$.  
     The sign vector  $(p_{12},p_{14},p_{23},p_{34})=(+,+,-,-)$
    of $V$ and $W$ are equal. By Theorem~\ref{thm : sign var} we know that the signs of $V$ and $W$ on $p_{13},p_{24}$ must be different.
    This is indeed the case. See Figure~\ref{fig : bad example} for a visual representation of a move from $V$ to $W$.
    \begin{figure}
        \centering
        \includegraphics[width=0.5\linewidth]{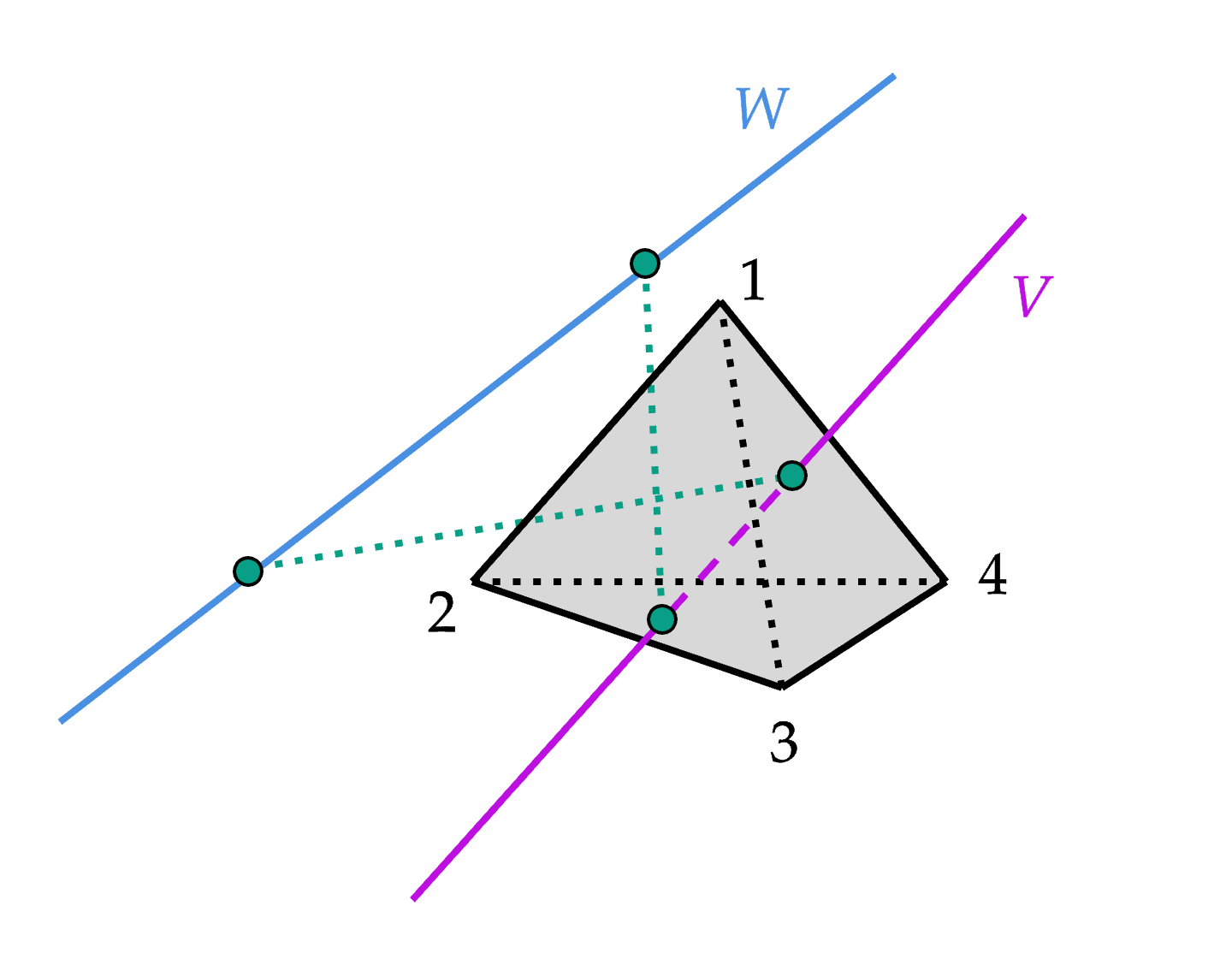}
        \caption{A visual representation of Example~\ref{ex : bad example}. We can see that the line $V$ stabs the simplex and the line $W$ doesn't. Moreover we can see the points that the path constructed in the proof of Theorem~\ref{thm : sign var} would follow. These exit the polytope without crossing any divisor in the Schubert arrangement considered.}
        \label{fig : bad example}
    \end{figure}
\end{example}
\end{rem}

Since Schubert arrangements arising from polytopes are not generic, we expect them to have nice properties. In particular, we conjecture that the following holds.
\begin{conj}
    Given a full-dimensional polytope $P\subseteq \PP^{n-1}$, the number of cells in its $k$-face Schubert arrangement inside of $\Pk$ is equal to the number of sign vectors realized by $V\in \Pk\setminus \HPk$.
\end{conj}

\subsection{Recovering $\Pk$ from $\Pkmax$}

This section is dedicated to the study of the $k$-stabbing set of a polytope $P$ and its relation to the maximally $k$-stabbing set. In particular, we can recover all points intersecting $P$ in small dimension faces from the closure of the maximally $k$-stabbing set.

\begin{prop}\label{thm : closure}
    Let $P\subseteq \PP^n$ be a full-dimensional polytope and fix $0\leq k \leq n$. Then 
    \[ \Pk = \cl(\Pkmax). \]
\end{prop}
\begin{proof}
    First, if $V\in \cl(\Pkmax)$, then clearly $V$ will intersect the polytope $P$, that is, $V\in \Pk$.

    Conversely, suppose $V\in \Pk \setminus \Pkmax$, that is, $V$ intersects $P$ in at least one face $F$ of dimension smaller than $n-k$. Consider two distinct cases.
    In the first case, suppose that $V\cap P$ has infinitely many points. Choose $k$ linearly independent points $u_1,\ldots,u_k$ on $P\cap V$, so that $V$ is the span of $u_1,\ldots, u_k$. Let $p\in \relint(P)$. Then the space 
    \[ U_\epsilon = \rowspan\begin{pmatrix}
        u_0 +\epsilon p\\ \vdots \\ u_k +\epsilon p
    \end{pmatrix} \in \Gr(k,n)\]
    for $0<\epsilon\ll 1$ and $U_\epsilon \in \Pkmax$ for a choice of $p$ generic enough. Since $\lim_{\epsilon \to 0} U_\epsilon = V$, it follows that $V$ is in the closure of $\Pkmax$.

    In the second case, suppose that $V\cap P=\{ u_1\}$. Since $P$ is full-dimensional polytope, there exists a sequence of spaces $W_m$ intersecting $P$ in infinitely many points and converging to $V$ as $m$ goes to infinity. Since by the previous step each of the $W_m$ lies in $\cl(\Pkmax)$ and the closure is idempotent, then $V\in \cl(\Pkmax)$.
\end{proof}

The closure operations translates to the sign description as explained in the following theorem.

\begin{prop}\label{thm : sign closure}
    Let $P\subseteq \PP^n$ be a full-dimensional polytope and fix $0\leq k \leq n$.
    For every $V\in \Gr(k,n)$, we have that $V\in \Pk$ if and only if there exists $W\in \Pkmax$ such that 
    \[ \sign(\CPk(V))\leq \sign(\CPk(W)). \]
\end{prop}
\begin{proof}
    If $V\in \Pk$, then $\sign(\CPk(V))\leq \sign(\CPk(W))$ for~some~$W\in\Pkmax$.

    Conversely, let $V$ and $W$ be such that $ \sign(\CPk(V))\leq \sign(\CPk(W))$ and $W\in\Pkmax$. If $V$ and $W$ have the same nonzero sign vector on the boundary of an $(n-k)$-dimensional face of $P$, then $V\in \Pk$, by proof of Theorem~\ref{thm : inequ}.
    If that is not true, then every $(n-k)$-dimensional face $F$ of $P$ has a facet $G_F$ such that $C_{G_F}(V)=0$. Since by Lemma~\ref{lemma : nonzero}, the vector of Chow form is not zero, we can choose a face $F$ of $P$ such that $\CPk(V)$ is not zero on every boundary of $F$, i.e. $V$ is not contained in the linear span of $F$.
    
    Now consider the path in $\Gr(k,n)$ obtained as follows. Let $v_1$ and $w_1$ be the intersections of $V$ and $W$ with $\spanv(F)$, respectively. Consider, similarly to proof of Theorem~\ref{thm : sign var}, the path defined by
     \[ U_t:=\begin{pmatrix}
        (1-t)v_1 + tw_1 \\
        w_2\\
        \vdots   \\
        w_{k} \\
        \end{pmatrix}, \]
    which is linear in Pl\"ucker coordinates. 
    The Chow form $\CPk(U_t)|_{G}$ is linear in the Pl\"uckers along the path $U_t$ and $\CPk(U_t)|_{G}=0$ exactly at $t=0$. 
     We now want to show that $v_1$ is contained in $G$.
     Assume $\spanv(G)\cap V = v_1$ is not inside $P$. Then, the path $ (1-t)v_1 + tw_1 $ in $\PP^n$ is contained in $F$ at $t=1$ and not contained at $t=0$. As the face $F$ is a convex polytope itself, the path passes through a facet $G'$ of $F$ at some $t_0 \in (0,1)$. 
     The chow form of $G'$ evaluated at $U_t$ is also linear and evaluates to $0$ at $t_0$. Thus, $\CPk(U_t)|_{G'}$ changes sign along $U_t$. 
     Extending the path to $\Tilde{U_t}$ as done in theorem 3.1. by moving the remaining $w_i$ to $v_i$, the Chow form $\CPk(\Tilde{U_t})|_{G'}$ does not change sign along this path, as it is piecewise linear and never zero. 
     Hence, $\CPk(V)|_{G'}$ and $\CPk(W)|_{G'}$ have different sign, contradicting the assumption.
     We conclude that $v_1$ is in $G$ and thus $V \in \Pk$.
\end{proof}

Not all the sign vectors obtained from a $\sign(\CPk(W))$ for some $W\in \Pkmax$ by adding some zeroes will be realized. A trivial example is the zero vector that by Lemma~\ref{lemma : nonzero} is never realizable. Another example is a Chow vector which is zero on all faces bounding two faces such that the intersection of their linear span has dimension smaller then $k$.
Note that for the stabbing sets of simplices, studying the realizability of sign vectors of Chow forms of $P$ corresponds to studying realizablity of oriented matroids. 

\section{Equations defining $\Pk$}

Using the properties of the $k$-stabbing set of $P$ studied in the previous section we can construct equations defining $\Pk$ as a semialgebraic subset of $\Gr(k,n)$. 

The description we want to define relies on the fact that, in order to know whether a $k$-subspace $V$ stabs the polytope we need to check less conditions then the ones prescribed by Theorems~\ref{thm : sign var} and~\ref{thm : sign closure}. In fact, in order to describe the spaces stabbing a given $(n-k)$-dimensional face of $P$, we only need to fix the signs vector of the Chow forms of the boundaries of that face.

\begin{thm}\label{thm : inequ}
    Let $P,k,n$ be fixed as in Setup~\ref{setup : Chow forms}. A $k$-subspace $V\in \Gr(k,n)$ intersects the polytope $P$ if and only if there exists $W\in \Pkmax$ and an $(n-k)$-dimensional face $F$ of $P$ such that
    \[ \sign(\CPk(V)|_F) \leq \sign(\CPk(W)|_F). \]
\end{thm}
\begin{proof}
    If $V$ intersects the polytope $P$ then by Theorem~\ref{thm : sign closure} we know that there must exist such a space $W\in \Pkmax$.

    Conversely, suppose that given $V\in \Gr(k,n)$ there exists $W\in \Pkmax$ such that
    $\sign(\CPk(V)|_F) \equiv \sign(\CPk(W)|_F)$, for some $(n-k)$-dimensional face $F$ of $P$. Suppose by contradiction that $V$ does not stab $P$. Let $\mF_P(W)$ be the family of $(n-k)$-dimensional faces of $P$ intersected by $W$. 
    Let $\mF_P(W) = \{ F=F_1, F_2, \ldots, F_r\}$ and define $w_i = F_i \cap W$ and $v_i = \spanv(F_i)\cap W$. 
    Construct matrices $U_t^i$ as in the proof of Theorem~\ref{thm : sign var}, which define a path from $W$ to $V$. 
    Along this path there exists $i_0\in \{0,\ldots,k\}$ and $t_0$ such that $U_t^i$ intersects $P$ for $t<t_0$ and $i\leq i_0$ and $U_t^i$ does not intersect $P$ for $t>t_0$ and $i\geq i_0$. 
    We have $i_0>1$ since in the first move we do not change any sign, hence we stay in the same stabbing chamber. Moreover, for $i>1$, the position of the space $U_t^i$ with respect to the face $F$ is not changed. Hence the spaces along the path $U_t^i$ can only exit the polytope through the face $F$, that is, $U_{t_0}^{i_0}\cap P$ is contained in $F$. This is only possible if some of the Chow forms of the boundaries of $F$ go to zero on $U_{t_0}^{i_0}$ and eventually change sign. This is against the assumption that $V$ and $W$ have the same signs along the boundaries of $F$, hence $V$ stabs the polytope.

    Finally, if $W\in \Pkmax$ is such that $\sign(\CPk(V)|_F) < \sign(\CPk(W)|_F)$, for some $(n-k)$-dimensional face $F$ of $P$, then applying the previous point we obtain that $V$ intersects $P$ in the face $F$ with arguments similar to the ones given in the proof of Theorem~\ref{thm : sign closure}. Hence $V\in \Pk$.
\end{proof}

Given a full-dimensional polytope $P\subseteq \PP^n$, the following procedure produces the set of inequalities defining the set $\Pk$ inside the Grassmannian $\Gr(k,n)$.

\begin{procedure}
    Let $P\subseteq \PP^n$ be a full-dimensional polytope and fix $1\leq k \leq n$. Define the Chow forms of $P$ according to Setup~\ref{setup : Chow forms}.
    \begin{enumerate}
        \item For each $(n-k)$-dimensional face $F$ of $P$ define the point $v_F$ as the sum of the vertices of $F$. In particular $v_F\in F$.
        \item Let $F_1,\ldots, F_k$ be $(n-k)$-dimensional faces of $P$ that are not contained in a $k$-dimensional face of $P$. Evaluate the vector of Chow forms $\CPk$ on $V_{F_1,\ldots,F_k}=\spanv(v_{F_1},\ldots, v_{F_k})$.
        The set of $V\in \Gr(k,n)$ that intersect $P$ in $F_i$ is given by 
        \[ \{ V\in \Gr(k,n) \mid \sign(\CPk(V)|_{F_i}) \leq \sign(\CPk(V_{F_1,\ldots,F_k})|_{F_i}) \}. \]
        \item Repeat the process until all the $(n-k)$-dimensional faces of $P$ have been considered. The set $\Pk$ is the union of these semialgebraic sets.
    \end{enumerate}
\end{procedure}

\begin{rem}
    If $k=1$, we are interested in studying the set of points intersecting the polytope. This construction recovers the hyperplane description of a polytope. 

    More interestingly, if $k=n-1$, then we recover the construction introduced in~\cite{BestPolytopeSlice}. In this case the stabbing chambers correspond to slicing chambers. Brandenburg, De Loera and Meroni studied optimization problems on what is the best way to slice a polytope.
\end{rem}

\begin{exa}\label{ex : octahedron}
    Consider the octahedron $P = \conv(v_{12}, v_{13}, v_{14}, v_{24}, v_{23}, v_{34})\subseteq \PP^3$ where $v_{ij} = e_i + e_j$, for $\{e_i\}_{i\in [4]}$ standard basis vectors. Fix $k=1$, that is, we want to study the set of lines intersecting the octahedron.

    The first set is to define the Chow forms we are interested in studying. Associate to an edge $v_{ij}-v_{ik}$ with $j<k$ the Chow form defined by the matrix $E^{ij}_{ik}$ with rows $v_{ij}, v_{ik}$. That is, if $C^{ij}_{ik} = \sum_{I\in \binom{[4]}{2}} q_I(E^{ij}_{ik})p_I$,  the vector of Chow forms of the octahedron is defined by
    \begin{align*}
        \CPk &= \left(C^{12}_{13},\, C^{12}_{14},\, C^{12}_{23}, \, C^{12}_{24},\, C^{13}_{23},\, C^{13}_{14},\, C^{13}_{34},\, C^{14}_{24},\, C^{14}_{34},\, C^{23}_{24},\, C^{23}_{34},\, C^{24}_{34}\right).
    \end{align*}
    If we are interested in the set of lines passing through the face defined by $v_{12}, v_{13}, v_{23}$ then we can compute the sign vector of the Chow~forms~evaluated~at
    \[ V = \rowspan \begin{pmatrix}
        2 & 2 & 2 & 0\\
        0 & 2 & 2 & 2
    \end{pmatrix}. \]
    This gives rise to the following sign vector:
    \[ \left(\mathcolorbox{blue!30}{-},\, -,\, \mathcolorbox{blue!30}{+}, \, 0,\, \mathcolorbox{blue!30}{-},\, +,\, 0,\, +,\, -,\, \mathcolorbox{red!40}{-},\, \mathcolorbox{red!40}{+},\, \mathcolorbox{red!40}{-}\right), \]
    where the signs highlighted in blue correspond to boundaries of the face defined by $v_{12}, v_{13}, v_{23}$ and the signs highlighted in red correspond to boundaries of the face defined by $v_{23}, v_{24}, v_{34}$, the other face stabbed by $V$. See Figure~\ref{fig : octahedron}.
    \begin{figure}
        \centering
        \includegraphics[width=0.5\linewidth]{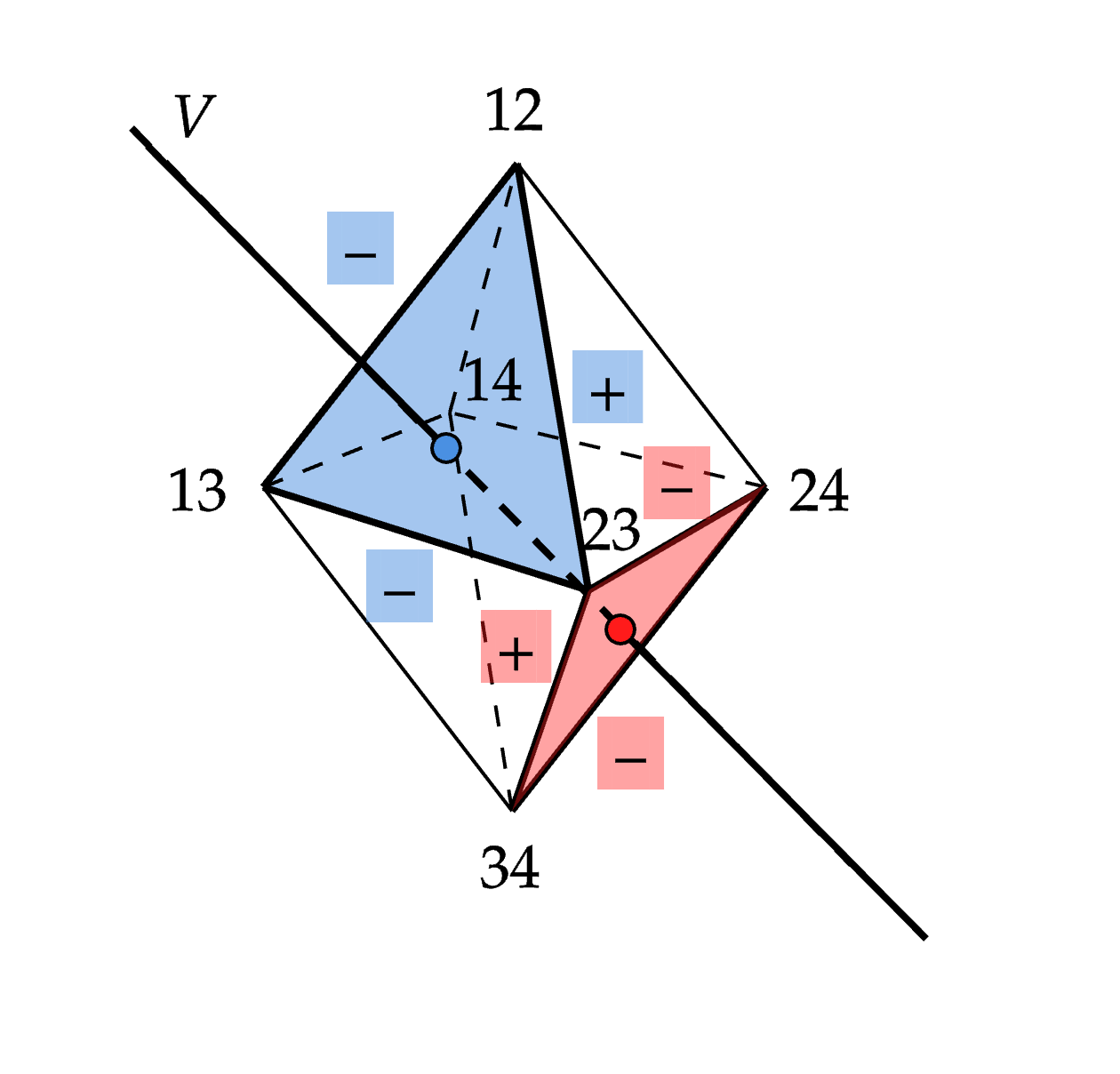}
        \caption{The line $V$ from Example~\ref{ex : octahedron} stabbing the octahedron and the sign condition on the boundaries of the stabbed faces.}
        \label{fig : octahedron}
    \end{figure}

    The set of lines stabbing the octahedron in the face $\conv(v_{12}, v_{13}, v_{23})$ is exactly the set of $V\in \Gr(k,n)$ such that $\sign\left(\CPk(V)|_{E^{12}_{13},E^{12}_{23},E^{13}_{23}}\right) \equiv (-,+,-)$. This description extends to every facet of the octahedron, as proved in Proposition~\ref{prop : simplicial}.
\end{exa}

\section{How to stab a simplicial polytope}\label{sec : simplicial}

We now focus on studying the $k$-stabbing sets of a special class of polytopes, called simplicial polytopes.
\begin{dfn}
    Let $1 \leq \ell < n$. A polytope $P\subseteq \PP^{n-1}$ is said to be $\ell$-simplicial if all the $\ell$-dimensional faces of $P$ are simplices.
\end{dfn}
In particular, cyclic polytopes, simplicial polytopes and $(n-k)$-neighborly polytopes are $(n-k)$-simplicial.
Spaces of dimension $k$ stabbing an $(n-k)$-simplicial polytope have the following explicit characterization. 

\begin{prop}\label{prop : simplicial}
    Let $1\leq k \leq n$ and let $P\subseteq \PP^{n-1}$ be a full-dimensional $(n-k)$-simplicial polytope. Suppose $P$ has $m$ vertices $v_1,\ldots,v_m$.
    If $F$ is an $(n-k)$-dimensional face of $P$ given by the convex hull of some vertices indexed by $S\in \binom{[m]}{n-k+1}$, with $S=\{ s_1< \ldots< s_{n-k+1} \}$, then $V\in \Gr(k,n)$ intersects $F$ if and only if 
    \[ \sign(C_{S\setminus s_{n-k+1}}(V), C_{S\setminus s_{n-k}}(V), \ldots, C_{S\setminus s_1}(V)) \equiv (+,-,\ldots, (-1)^{n-k}) \]
\end{prop}
\begin{proof}
    Let $F_1,\ldots, F_{k-1}$ be faces of $P$ such that $F,F_1,\ldots,F_{k-1}$ do not lie on a facet of $P$. By Theorem~\ref{thm : inequ}, $V$ intersects $P$ in $F$ if and only if the Chow forms of the facets of $F$ evaluated on $V$ have the same sign as $W = \spanv(v_F, v_{F_1},\ldots, v_{F_{k-1}})$. Represent $W$ with a matrix containing the above points in its rows in this order. Then for any $1\leq i \leq n-k+1$ we have that
    \[ C_{S\setminus s_{i}}(W) = \det\begin{pmatrix}
        v_{s_i}^t | v_{F_1}^t | \ldots | v_{F_{k-1}}^t | v_{s_1}^t| \ldots | v_{s_{i-1}}^t| v_{s_{i+1}}^t| \ldots | v_{s_{n-k+1}}^t
    \end{pmatrix}. \]
    The Chow forms $C_{S\setminus s_{i}}(W)$ and
    \[\det\begin{pmatrix}
         v_{F_1}^t | \ldots | v_{F_{k-1}}^t | v_{s_1}^t| \ldots | v_{s_{n-k+1}}^t
    \end{pmatrix}.\]
    differ by a sign factor given by $(-1)^{k+i-1}$. Since the second term does not depend on $i$, we obtain the sign characterization in the statement.
\end{proof}

If $k=n-1$, then we can apply the description given by Proposition~\ref{prop : simplicial} to any polytope. As a consequence, an hyperplane $H \in \Gr(n-1,n)$ stabs a polytope $P$ in an edge $E = \conv(v_1,v_2)$ if and only if 
\[\sign(C_{v_1}(H),C_{v_2}(H)) = (+,-).\] 
This corresponds to requiring that the vertices $v_1$ and $v_2$ are on two different half-spaces defined by $H$, as expected.
\begin{cor}\label{rem : slicing}
    The slicing set of a polytope $P$ is
    \[ P^{[n-1]} = \{ V\in \Gr(n-1,n) \mid \overline{\var}(\mathrm{C}_P^{n-1}(V))\geq 1 \}. \]
\end{cor}

\subsection{Amplituhedra and stabbing sets}\label{sec : stabbing cyclic}

We will now study the relation between amplituhedra and stabbing sets of polytopes. The first example we will consider is that of the totally non-negative Grassmannian. This set can be described as on of the stabbing chambers of the standard simplex, given by spaces intersecting faces defined by consecutive vertices.
\begin{prop}\label{prop : nonneg Grass}
    The following holds:
    \[ \Gr_{\geq 0}(k,n) \subseteq \cl\left(\mS_{\Delta_n}\begin{pmatrix}
        e_1+e_2+\cdots + e_{n-k+1}\\
        \vdots\\
        e_k + e_{k+1}+ \cdots + e_n
    \end{pmatrix} \right), \]
    where $e_1, \ldots, e_n$ are the canonical basis of $\RR^n$.
\end{prop}
\begin{proof}
    Let $W$ be the $k$-dimensional space defined by the matrix in the right-hand side of the equality and let $V\in \Gr_{\geq 0}(k,n)$. The vector of Chow forms of the simplex $\Delta_n$ is given by $((-1)^{\sign(I)}q_I)_{I\in \binom{[n]}{k}}$. Therefore $\sign(\mathrm{C}_{\Delta_m}^k(V)) \leq \left((-1)^{\sign(I)}\right)_{I\in \binom{[n]}{k}}$. Let $F_i$ be the $(n-k)$-dimensional face of $\Delta_n$ given by the convex hull of $e_i,e_{i+1},\ldots, e_{i+n-k}$. By construction $W$ stabs $\Delta_n$ in $F_1, \ldots, F_k$. By Proposition~\ref{prop : simplicial}, we obtain that 
    \[(\sign(C_{F_i\setminus \{v_j\}}(W)))_{j=i,\ldots,n-k+i} = (-1)^{\sign(\{i, \ldots, n-k+i\}\setminus \{j\})}.\] 
    Therefore the sign vector of $V$ and $W$ coincide on the boundaries of the faces $F_1, \ldots, F_k$, hence $V \in \cl(\mS_{\Delta_n}(W))$.
\end{proof}

\begin{rem}
    In general the inclusion is strict. For example the set of lines stabbing the $3$-dimensional simplex in the faces $\conv(e_1,e_2,e_3)$ and $\conv(e_2,e_3,e_4)$ is given by points whose non-zero Pl\"ucker coordinates different from $q_{23}$ have the same sign.
\end{rem}

For more general amplituhedra, we will obtain that they are subsets of the stabbing set of a cyclic polytope.
Recall that cyclic polytopes are polytopes isomorphic to the convex hull of $m$ points on the $n$-dimensional moment curve. The combinatorial equivalence class is the same independently of the choice of points on the moment curve, hence any such polytope is denoted by $C(m,n)$. $C(m,n)$ is a full-dimensional polytope in $\PP^{n-1}$ with a natural order on the vertices given by their order on the moment curve.
Any polytope with an order on the vertices such that the corresponding matrix lies in the totally positive Grassmannian is a cyclic polytope~\cite{SturmfelsCyclic}.

A complete description of the combinatorial structure of the cyclic polytope is given by Gale's evenness condition~\cite{gale}.
In particular, the cyclic polytope $C(m,n)$ is a simplicial polytope. Moreover, any $n$-subset $S\subseteq [m]$ forms a facet of $C(m,n)$ if and only if $S$ satisfies Gale's evenness condition: if $i<j$ are in $[m]\setminus S$, the number of $k\in S$ with $i<k<j$ is even. 

\vspace{5pt}

Recall that a point in the amplituhedron is given by $C\cdot Z$ where $Z$ is the fixed totally positive $n\times(k+m)$ matrix and $C\in \Gr_{\geq 0}(k,n)$. Since $C$ stabs the simplex $\Delta_n$, we can represent $C$ with a $k\times n$ matrix containing only non-negative entries in its first row. Hence $C\cdot Z$ contains in its first row a point in the cyclic polytope defined by $Z$, i.e. $\mA_{n,k,m}\subseteq C(n, k+m)^{[k]}$.

The amplituhedron has been described in terms of sign characterization in~\cite{unwinding}. There it is conjectured that $\mA_{n,k,m}$ is equal to the set $A_{n,k,m}$ of spaces $V\in \Gr(k,k+m)$ such that:
\[ \begin{cases}
(C_{12\ldots (m-1)m}(V), \ldots ,C_{12\ldots(m-1)n}(V)) \text{ has } k \text{ sign flips,}\\
(-1)^kC_{1(i_1\,i_1 +1)\ldots(i_{m_2}\, i_{m_2}+1)}(V)>0, \; C_{(i_1\,i_1 +1)\ldots(i_{m_2}\, i_{m_2}+1)n}(V)>0 \; \text{ if } m \text{ is odd} \\
C_{(i_1\,i_1 +1)\ldots\left(i_{m_2}\, i_{m_2}+1\right)}(V)>0 \; \text{ if } m \text{ is even}
\end{cases}
\]
where $m_2 = \lfloor{\frac{m}{2}}\rfloor$ and all indices are distinct. In the context of the amplituhedron, Chow forms defined by faces of the cyclic polytope are (a subset of) {\em twistor coordinates} and they are denoted by $\langle i_1\ldots i_m \rangle = C_{i_1\ldots i_m}$.
The sign description has been proven to be equivalent to the original definition of the amplituhedron only in the cases $m=1$~\cite{m=1Amplituhedron}, $m=2$~\cite{m2signs}. 

\begin{rem}\label{amplituhedron stabbing}
    Combining the sign description description of the amplituhedron and Proposition~\ref{prop : simplicial} we obtain that if $m=2$ and  $V\in A_{n,k,2} = \mA_{n,k,2}$, then $V$ stabs the cyclic polytope $C(n,k+2)$ in $\left\lceil\frac{k}{2}\right\rceil$ faces of the form $\conv(v_1,v_i, v_{i+1})$ with $1<i<n-1$.
\end{rem}

The requirement for the sign flips in the definition of $A_{n,k,m}$ to be exactly $k$ is maximal, that is, it is not possible to have more than $k$ sign flips. Spaces defined by requiring a smaller amount of sign flips have been studied for $m=4$ in~\cite{gabriele}.

\section{Open problems}

There are still many open problem related to the $k$-stabbing set of a polytope and its relations to positive geometries. Below, we highlight two directions that we think would be particularly interesting to explore in future research.

\vspace{5pt}
\noindent{\bf Polytopal Schubert arrangements.}
Arrangements of Schubert divisors from faces of polytopes are highly non-generic due to the incidence relations among the faces of the polytope.
We therefore expect them to have nicer properties than generic hypersurface arrangements. Examples from \cite{Schubertarr} show that the regions in a Schubert arrangement might not be contractible, and that different regions can have the same sign vector, see Example~\ref{ex : bad example}.

\begin{que}
    What can be said about the topology of connected regions of the $k$-face Schubert arrangement of a polytope? Is there a combinatorial interpretation for the Euler characteristic of such a Schubert arrangement? Is the number of sign vectors realized by points in $\Gr(k,n)\setminus\HPk$ equal to the number of connected regions of the arrangement?
\end{que}

Breiding, Sturmfels and Wang~\cite{computingarr} recently introduced an algorithm based on Morse theory to compute regions in the complement of a hypersurface arrangement. This has been implemented in {\em Julia}~\cite{julia}, in the package chambers.jl.

\begin{que}
    Is there a more efficient algorithm to compute the number of components in the $k$-face Schubert arrangement of a polytope?
\end{que}

\noindent{\bf Positive geometries.}
The relations between amplituhedra and the $k$-stabbing set $\Pk$ leads to a natural question: Is the set $\Pk$ a positive geometry? In order to answer this question, a deep understanding of the boundary structure of $\Pk$ is necessary. This structure is related to the face poset of the polytope $P$. For example, for lines stabbing a $3$-dimensional polytope, codimension $1$ boundaries are given by lines stabbing the polytope in one point of an edge, codimension $2$ boundaries are given by lines lying on a $2$-dimensional face of the polytope, codimension $3$ boundaries are lines through an edge and a vertex of the same face and, finally, codimension $4$ boundaries are given by the lines spanned by edges. More generally, an answer to the following questions would be necessary.

\begin{que}
    What is the face stratification of the set $\Pk$ and how is it related to boundary poset of $P$? What would the canonical form of $\Pk$ be and how can we compute it? How would this canonical form be related to the canonical form of the polytope or the loop amplituhedron?
\end{que}


\section*{Acknowledgements} The authors would like to thank Bernd Sturmfels for introducing them to the problem and providing the data to solve the question in specific examples. Moreover the authors are grateful to Evgeniya Akhmedova, Gabriele Dian, Elia Mazzucchelli, Rémi Prébet, Kexin Wang for many useful discussions and explanations. S.S. also thanks Erdenebayar Bayarmagnai, Emiliano Liwski, Fatemeh Mohammadi, Matthias Stout.

\begin{footnotesize}
\noindent {\bf Funding statement}:
Sebastian Seemann was funded by FWO fundamental research fellowship number (11PEP24N).
Francesca Zaffalon was partially supported by  FWO fundamental research fellowship (1189923N), FWO long stay abroad grant (V414224N) and by the European Union (ERC, UNIVERSE PLUS, 101118787).
Views and opinions expressed are however those of the authors only and do not necessarily reflect those of the European Union or the European Research Council Executive Agency. Neither the European Union nor the granting authority can be held responsible for them.
\end{footnotesize}

\input{bibliography}

\end{document}

%% file: bibliography.tex
\begin{small}

\end{small}

%% file: main.bbl
\begin{thebibliography}{10}
\setlength{\itemsep}{-0.1mm}

\bibitem{PositiveGeometryandCanonicalForms}
N.~Arkani-Hamed, Y.~Bai and T.~Lam:
    {\em Positive Geometries and Canonical Forms} Journal of High Energy Physics {\bf No. 11} (2017)

\bibitem{unwinding}
N.~Arkani-Hamed, H.~Thomas, and J.~Trnka:
    {\em Unwinding the Amplituhedron in Binary}, Journal of High Energy Physics {\bf No. 16} (2018) 

\bibitem{amplituhedron} 
N.~Arkani-Hamed and J.~Trnka: 
    {\em The Amplituhedron}, Journal of High Energy Physics {\bf No. 10} (2014)

\bibitem{julia}
J.~Bezanson, A.~Edelman, S.~Karpinski, and V.~B.~Shah,: {\em Julia: A fresh approach to numerical computing}, SIAM review {\bf No. 59(1)} (2017)

\bibitem{BestPolytopeSlice}
M.~Brandenburg and J.~A.~De~Loera and C.~Meroni:
    {\em The Best Ways to Slice a Polytope}, 
    Mathematics of Computations (2024)
    
\bibitem{computingarr}
P.~Breiding, B.~Sturmfels, and K.~Wang: 
    {\em Computing Arrangements of Hypersurfaces}, arXiv:2409.09622

\bibitem{chow}
J.~Dalbec, and B~ Sturmfels:
    {\em Introduction to Chow forms}, Invariant Methods in Discrete and Computational Geometry: Proceedings of the Curaçao Conference, 13–17 June, 1994 (1995)

\bibitem{gabriele}
G.~Dian, and P.~Heslop: 
    {\em Amplituhedron-like geometries}, Journal of High Energy Physics {\bf No. 11} (2021)

\bibitem{Even-Zohar:2023del}
C.~Even-Zohar, T.~Lakrec, M.~Parisi, R.~Tessler, M.~Sherman-Bennett and L.~Williams:
    {\em Cluster algebras and tilings for the m=4 amplituhedron}, 
    arXiv:2310.17727
    
\bibitem{fulton}
W.~Fulton:
    {\em Young tableaux: with applications to representation theory and geometry}, Cambridge University Press {\bf No. 35} (1997)

\bibitem{gale}
D.~Gale:
    {\em Neighborly and cyclic polytopes}, Proc. Sympos. Pure Math. {\bf Vol. 7. No. 7} (1963).

\bibitem{gantkrein}
F.~R.~Gantmakher and M.~G.~Krein. 
    {\em Oscillyacionye matricy i yadra i malye kolebaniya mehaniˇceskih sistem} Gosudarstv. Isdat. Tehn.-Teor. Lit., Moscow-Leningrad, (1950)



\bibitem{karp}
S.~Karp:
    {\em Sign variation, the Grassmannian, and total positivity}, Journal of Combinatorial Theory, Series A {\bf No. 145} (2017)

\bibitem{m=1Amplituhedron}
S.~Karp and L.~Williams:
    {\em The m=1 amplituhedron and cyclic hyperplane arrangements},
    International Mathematics Research Notices {\bf No. 8} (2019)

\bibitem{LamGrasstopes}
T.~Lam: 
    {\em Totally nonnegative Grassmannian and Grassmann polytopes},
    Current Developments in Mathematics {\bf No.1} (2014)

\bibitem{m=2ParisiWilliamsLuk}
T.~Lukowski, M.~Parisi and L.~Williams:
    {\em The Positive Tropical Grassmannian, the Hypersimplex, and the m = 2 Amplituhedron},
   International Mathematics Research Notices
   {\bf No. 3} (2023)

\bibitem{grasstopes}
Y.~Mandelshtam, D.~Pavlov and E.~Pratt:
    {\em Combinatorics of $ m= 1$ Grasstopes}, arXiv:2307.09603

\bibitem{Schubertarr} 
E.~Mazzucchelli, D.~Pavlov and K.~Wang:
    {\em Hyperplane Arrangements in the Grassmannian},
    arXiv:2409.04288

\bibitem{m2signs}
M.~Parisi, M.~Sherman-Bennett and L.~Williams: 
    {\em The $m=2$ amplituhedron and the hypersimplex: signs, clusters, triangulations, Eulerian numbers},
    arXiv:2104.08254

\bibitem{Simonampl}
K.~Ranestad, R. Sinn, and S. Telen: 
{\em Adjoints and Canonical Forms of Tree Amplituhedra}, arXiv:2402.06527

\bibitem{SturmfelsCyclic}
B. ~Sturmfels:
    {\em Totally positive matrices and cyclic polytopes}, Linear Algebra and its Applications {\bf No. 107} (1988)

\end{thebibliography}
